\documentclass[11pt]{amsart}
\pdfoutput=1
\usepackage{amsthm}
\usepackage{amsmath}
\usepackage{aliascnt}

\usepackage{amssymb, amsbsy, mathtools, bbm, amsfonts, verbatim, array, floatflt, nicematrix}

\usepackage{hyperref}
\usepackage{cleveref}

\usepackage[margin=1.2in]{geometry}
\usepackage[all]{xy}
\usepackage[english]{babel}
\usepackage[shortlabels]{enumitem}
\usepackage{stmaryrd} 

\usepackage{graphicx,overpic}
\usepackage[usenames,dvipsnames,svgnames,table]{xcolor}

\usepackage{multirow}
\usepackage{easybmat}
\usepackage{adjustbox}

\usepackage{soul}

\usepackage{tikz}

\theoremstyle{plain}
\newtheorem{thm}{Theorem}[section]
\newtheorem*{thm*}{Theorem}
\newaliascnt{proposition}{thm}
\newtheorem{prop}[proposition]{Proposition}
\aliascntresetthe{proposition}
\newaliascnt{lemma}{thm}
\newtheorem{lemma}[lemma]{Lemma}
\aliascntresetthe{lemma}
\newaliascnt{corollary}{thm}
\newtheorem{corollary}[corollary]{Corollary}
\aliascntresetthe{corollary}

\theoremstyle{definition}
\newaliascnt{example}{thm}
\newtheorem{example}[example]{Example}
\aliascntresetthe{example}
\newaliascnt{remark}{thm}
\newtheorem{remark}[remark]{Remark}
\aliascntresetthe{remark}
\newaliascnt{definition}{thm}
\newtheorem{definition}[definition]{Definition}
\aliascntresetthe{definition}

\newtheorem*{prop*}{Proposition}
\newtheorem*{lemma*}{Lemma}

\newcommand{\R}{\mathbb{R}}
\newcommand{\Z}{\mathbb{Z}}
\newcommand{\N}{\mathbb{N}}

\newcommand{\F}{\mathbb{F}}
\newcommand{\E}{\mathbb{E}}

\newcommand{\define}{\mathrel{\mathop:}=}

\newcommand{\id}{\mathbf{1}}
\newcommand{\Id}{\operatorname{I}} 
\newcommand{\aG}{G} 
\newcommand{\aH}{H} 
\newcommand{\sH}{{H_0}} 
\newcommand{\On}{\operatorname{O}(n)} 
\newcommand{\SO}{\operatorname{SO}} 

\newcommand{\Mov}{\textsc{Mov}} 
\newcommand{\Mod}{\textsc{Mod}} 
\newcommand{\Fix}{\textsc{Fix}} 
\newcommand{\Isom}{\operatorname{Isom}} 
\newcommand{\Cent}{\operatorname{C}} 

\newcommand{\Span}{\textsc{Span}}
\newcommand{\Range}{\operatorname{Im}}
\newcommand{\Ker}{\operatorname{Ker}}

\newcommand{\hconj}{[h]_\aH}
\newcommand{\coconj}{\operatorname{C}} 
\newcommand{\coconjH}[2]{\operatorname{C}_\aH(#1,#2)} 
\newcommand{\scoconjH}[2]{\operatorname{C}_{\sH}(#1,#2)} 

\newcommand{\cl}{\mathrm{c}\ell} 
\newcommand{\CLF}{\operatorname{CL}} 
\newcommand{\ctn}{\operatorname{ctn}} 
\newcommand{\Tnorm}{\operatorname{CT}} 

\newcommand{\LAb}{\mathcal{L}_{A,b} } 

\newcommand{\veps}{\varepsilon} 
\newcommand{\quer}[1]{\overline{#1}}

\newcommand{\til}[1]{\widetilde{#1}}

\newcommand{\mbf}{\mathbf}

\newcommand{\into}{\hookrightarrow}
\newcommand{\onto}{\twoheadrightarrow}
\newcommand{\nsgp}{\trianglelefteq} 
\newcommand{\qi}{\simeq_{\text{q.i.}}} 
\newcommand{\qdom}{\preccurlyeq} 
\newcommand{\qeq}{\simeq_{\mathrm{qe}}} 

\numberwithin{equation}{section}

\definecolor{amethyst}{rgb}{0.6, 0.4, 0.8}
\definecolor{kellygreen}{rgb}{0.3, 0.73, 0.09}
\definecolor{americanrose}{rgb}{1.0, 0.01, 0.24}

\begin{document}

\hypersetup{pdfauthor={Santos{ R}ego, Schwer},pdftitle={Conjugator length for groups of Euclidean isometries}}

\title[Conjugator length of locally compact groups of Euclidean isometries]{Conjugator length of locally compact groups of Euclidean isometries}

\author{Yuri Santos Rego}
\address{Yuri Santos Rego \\ 
University of Lincoln\\ 
Charlotte Scott Research Centre for Algebra\\ 
Brayford Pool, LN6 7TS Lincoln, UK}
\email{ysantosrego@lincoln.ac.uk}

\author{Petra Schwer}
\address{Petra Schwer, Institute of Mathematics, Heidelberg University, Im Neuenheimer Feld 205, 69120 Heidelberg, Germany}
\email{schwer@uni-heidelberg.de}

\thanks{This research was supported in part by the DFG research training group RTG2229 `Asymptotic Invariants and Limits of Groups and Spaces'.}




\begin{abstract} 
We consider locally compact subgroups $H$ of the full isometry group $\Isom(\E^n)$ of Euclidean $n$-space which respect the splitting into an orthogonal and a translation subgroup. We prove that the conjugator length function of such groups either has zero growth, grows linearly, or is unbounded --- depending on the topology of the spherical part of $H$. 
Our theorem shows, in particular, that affine Coxeter groups and split crystallographic groups have linear growth for their conjugator length functions. 
\end{abstract}

\maketitle

\begin{center}
    We dedicate this paper to Martin Bridson on the occasion of his 60th birthday.  
\end{center}



\section{Introduction}\label{sec:Intro}

The growth of the conjugator length function has been investigated for various classes of finitely generated groups for more than 30 years. 
There has been a particular interest in groups acting on nice spaces.  
As a worst case scenario, groups with properly discontinuous, cocompact actions on a CAT$(0)$-space have conjugator length functions growing at most exponentially  \cite{BridsonHaefliger}. 
It is a broad problem to pin down the exact growth rate of this function for prominent (classes of) groups.

In this article we extend the notion of conjugator length function to locally compact (Hausdorff) groups, see \Cref{sec:coarse-geometry}, adding to the list of groups for which this function grows linearly while also observing an interesting dichotomy of cases. See \Cref{sec:mainsetup} for the precise definition of split subgroups $T_\aH \rtimes \sH$ of $\Isom(\E^n)$,  
and \Cref{def:essentialH0} for the essential part $\til{\sH}$ of $\sH$, which characterizes the action of $\sH$ on the translation lattice $L_H$ associated to $T_H$. 

\begin{thm}
	\label{thm:lineargrowth}
	Let $\aH=T_H\rtimes \sH$ be a split, locally compact subgroup of $\Isom(\E^n)$. Let $\til{\sH}$ be the essential part of $\sH$, write $\CLF: \N \to \Z_{\geq 0}$ for the conjugator length function of $H$. Then:
	\begin{enumerate} 
		\item\label{trivialGrowth} The conjugator length function $\CLF$ has trivial (i.e., zero) growth if, and only if, the spherical part $\sH$ and the translation part $T_H$  of $\aH$ commute.
	\end{enumerate} 
	Suppose that $T_H$ and $\sH$ do not commute. Then: 
	\begin{enumerate}[resume] 
		\item\label{~H0discrete}\label{1isolated} 
		The conjugator length function $\CLF$ of $H$ grows linearly if, and only if, $\til\sH$ is discrete.
		\item\label{~H0nondiscrete}\label{1notisolated} 
		$\CLF$ is unbounded (in the sense that $\CLF(m)=\infty$ for all but finitely many $m \in \N$) if, and only if, $\til\sH$ is not discrete.
	\end{enumerate} 
\end{thm}

It is known to experts that the conjugator length function is not necessarily preserved under quasi-isometries. Our \Cref{thm:lineargrowth} furnishes a long list of (pairs of) familiar topological groups that are quasi-isometric but have distinct conjugator length functions.

Observe that a trivial action of $\sH$ on $\Span_\R(L_H)$ implies  splitting of $\aH$ as a direct product, and that discreteness of $\til\sH$ is equivalent to finiteness in our setting. 
If $\aH$ is abelian or compact, then the conjugator length function of $\aH$ has zero growth by item~\eqref{trivialGrowth}.
The following corollary of our main result is immediate from  item~\eqref{1isolated}. 

\begin{corollary}
The conjugator length functions of affine Coxeter groups and of split crystallographic groups grow linearly.
\end{corollary}

A surprisingly long list of results yield linear growth. The oldest such example we found is due to Lys\"enok \cite{Lysenok} for word hyperbolic groups. Other examples include mapping class groups \cite{Tao} and some $S$-arithmetic metabelian groups such as Baumslag--Solitar groups $\mathtt{BS}(1,p)$ and lamplighters $\F_p \wr \Z$ \cite{Sale}. The most recent result we are aware of is by Bridson, Riley, and Sale \cite{BRS} establishing linear growth within a class of free-by-cyclic groups. 
It is often very helpful if the conjugator length function has linear growth. 
As a consequence, for example, 
an exponential-time algorithm to solve the conjugacy problem can be built for groups that have decidable word problem and linear conjugator length function; see, e.g., \cite[p.~416]{Tao}.

Many other interesting classes of groups sport nonlinear growth. 
To name just a few: the growth is at most quadratic for prime $3$-manifold groups \cite{BehrstockDrutu,Sale} and for Thompson's group $F$ \cite{BelkMatucci}, and at most cubic in free soluble groups \cite{Sale-freeSolvable}. Moreover, Bridson and Riley recently showed that all polynomials can appear as conjugator length functions of finitely presented groups \cite{BridsonRiley}. The same authors also establish polynomial upper bounds on the lengths of minimal conjugators in 2-step nilpotent groups \cite{BridsonRiley2}.

\medskip

Let us comment on the proof of \Cref{thm:lineargrowth}. 
The main idea is to use the geometric characterization of set set of conjugating elements, called transporting set\footnote{In \cite{MST4} this set was called coconjugation sets}, as provided by \cite{MST4}. We proceed as follows: 
For a pair of conjugated elements $h,h'$ whose lengths sum up to at most $m$ we aim to control the length of a shortest possible conjugating element (alternatively referred to as \emph{transporter}) $k$. 
As a first step, we prove in \Cref{prop:reduce-to-translations} that one can consider the length of the translation vector $\lambda$ of $k=t^\lambda u$ instead of the word length of $k$. 
To estimate the norm of $\lambda$ we make use of a result in \cite{MST4}, restated as \Cref{thm:coconj}, ensuring that all possible $\lambda$ are contained in an affine subspace which appears as a set of solutions of a linear system of equations determined by $h$ and $h'$; see Equation~\ref{eq:eta}. 
We then resort to Moore--Penrose pseudoinverses, see \Cref{sec:pseudoinverse}, which yield explicit minimizing solutions to those systems, together with some group theoretic properties of $H$ summarized in \Cref{sec:mainsetup}. A straightforward construction of conjugated pairs allows us to rule out the case~\eqref{trivialGrowth} of no growth and obtain a general lower bound of linear growth; see \Cref{prop:Tnorm-lower-bound}.

The case of \Cref{thm:lineargrowth}\eqref{1isolated} resembles a scenario where the spherical subgroup $\sH$ is essentially discrete --- summarized under the formal assumption that $\til\sH$, as conceived in \Cref{def:essentialH0}, is discrete. This step involves estimates that are carried out in the proof of \Cref{prop:Tnorm-growth-discrete,prop:Tnorm-growth-essential-restriction}. 
See \Cref{ex:conjugatornorm} and \Cref{fig:conjugatornorm} for a picture that illustrates these ideas in the case of a Coxeter group $W$ of type $\widetilde{\mathtt{A}}_2$, in which the point group $W_0$ is finite.  
Otherwise, if $\til\sH$ is not discrete and $\CLF$ is not trivial, we prove that one may choose $h, h'$ in such a way that the norm of the translation vector of a transporter becomes arbitrarily large. We refer the reader to the proof of \Cref{prop:Tnorm-growth-nonisolated} for details.

The formal proof of our main theorem reads as follows. 
 
\begin{proof}[Proof of \Cref{thm:lineargrowth}]
	Item~\eqref{trivialGrowth} is just \Cref{cor:thecaseofzerogrowth}. 
	The other cases rely on \Cref{prop:Tnorm-lower-bound} which establishes a linear lower bound in both cases. Linear growth in case $\til\sH$ is discrete is covered by \Cref{cor:thecaseoflineargrowth}. 
	\Cref{prop:Tnorm-growth-nonisolated} implies that, in case $\til\sH$ is nondiscrete then the growth, if nontrivial, is unbounded. Because of the mutually exclusive cases on the conditions on $\sH$ the equivalences stated in items~\eqref{~H0discrete} and~\eqref{~H0nondiscrete} follow. 
\end{proof}

\subsection*{Acknowledgements} 
First, we would like to thank Martin Bridson for his beautiful talk in Karlsruhe in December 2024. It was this talk that motivated us to look into this question. Happy birthday, Martin! We also thank Karel Dekimpe, Linus Kramer, Bianca Marchionna, and Jos\'e Pedro Quintanilha 
for helpful discussions about locally compact and Lie groups. 
Further we owe thanks to the anonymous referee and to Lukas Vandeputte for their invaluable comments on an earlier version of this paper and pointing out an issue with the previous statement of \Cref{thm:lineargrowth}. 
We thank Vandeputte for drawing our attention to, and kindly supplying, \Cref{ex:Lukas}(2). Marc Burger suggested to follow the french tradition of Borel
 and others and write \emph{transporting set} in place of \emph{coconjugation set}. YSR thanks the Research Station Geometry + Dynamics in Heidelberg and their staff for their hospitality in May 2025 and March 2026. PS is grateful for being able to spend Spring 2025 at the IAS in Princeton.

\section{The groups we study and the tools we use}\label{sec:preliminaries}

In \Cref{sec:coarse-geometry} we extend the notion of conjugator length functions to the class of locally compact Hausdorff topological groups, collect some elementary properties of split subgroups of $\Isom(\E^n)$ in \Cref{sec:mainsetup}, recall geometric results on conjugation from \cite{MST4} in \Cref{sec:coconjugation-sets}, and finally cite from the literature some basic facts on pseudoinverses in \Cref{sec:pseudoinverse}. 

\subsection{Conjugator length and coarse geometry of locally compact groups}
\label{sec:coarse-geometry}

Interest in the conjugator length function for finitely generated (discrete) groups has been rapidly increasing in recent years. Here we briefly note that, similarly to Dehn functions, conjugator length is equally meaningful for compactly generated groups, so let us provide definitions. We disclaim that all topological groups we treat are Hausdorff.

Let $G$ be a locally compact topological group, and suppose $G$ is compactly generated. That is $G = \langle S \rangle$ for some compact subset $S \subseteq G$. The word metric $d_S$ for $G$ with respect to $S$ is geodesically adapted (see \cite[Chapter~4B]{Cornulier--delaHarpe}) and enjoys properties akin to the word metric of a finitely generated group and  a similar theory of (coarse) geometry for $G$ is available; we refer the reader to \cite{Cornulier--delaHarpe} for details. We may then consider, as usual, the corresponding word length $\ell_S = d_S(-,\id)$, the distance to the identity $\id \in G$.

The \emph{conjugator length} $\cl(h,h')$ of two conjugate elements $h,h'\in G$ measures the length of a shortest possible transporter. That is,  
\[
\cl(h,h')\define \min \{ \ell_S(k) : khk^{-1}=h' \text{ in } G \}. 
\]
Note that for compactly generated locally compact groups the minimum is always attained. 

\begin{definition}
\label{def:conjugatorlengthfunction}
The \emph{conjugator length function} of the compactly generated, locally compact group $G$ (a priori with respect to a compact generating set $S$) is the map
\[
\CLF: \N \to \Z_{\geq 0}, \quad \CLF(m) \define \sup\{ \cl(h,h') \, : \,  h\sim h', \, \text{ and } \, \ell_S(h)+\ell_S(h')\leq m\},
\]
where $\sim$ is the conjugation relation. 
We adopt the convention that $\CLF$ is constant equal to zero if $G$ is abelian. 
\end{definition}

Given nondecreasing functions $f,g : \R_{\geq 0} \to \R_{\geq 0}$, we say that \emph{$g$ quasi-dominates $f$}, written $f \qdom g$, if there exists an increasing affine function $a : \R \to \R$ such that $f(r) \leq a\circ g\circ a(r)$ for all $r\in\R$ with $a(r) \geq 0$. (By increasing affine function on $\R$ we mean a map $a(r)=Cr+D$ for some $C, D\in\R$ with $C > 0$.) If $f \qdom g$ and $g \qdom f$, we call $f$ and $g$ \emph{quasi-equivalent} and write $f \qeq g$. That $\qeq$ is an equivalence relation is an easy exercise. The \emph{growth rate} of $f$ is its equivalence class under $\qeq$, and we adopt common terminology. For instance, $f$ grows polynomially of degree $d$ in case $f\qeq (r \mapsto r^d)$, or exponentially when $f \qeq (r \mapsto \alpha^r)$ for some $\alpha > 1$, or has zero (or trivial) growth when $f \qeq (r \mapsto \text{constant})$. We remark that $f$ growing at least linearly is the same as $f$ quasi-dominating an increasing affine map. 

Note that $\CLF$ can be interpreted in an obvious way as a nondecreasing function from $\R_{\geq 0}$ to $\R_{\geq 0}$: set $\CLF(r) = \CLF(\lceil r \rceil)$ for $r \in \R_{\geq 0}$. We can then talk about its growth rate as above.

\begin{example} \label{ex:compactCLFzero}
If $G$ is compact, take $S = G$. Then $\CLF \qeq \mathbf{0}$, where $\mathbf{0} : \R_{\geq 0} \to \R_{\geq 0}$ is the zero map. Indeed, $\cl(h,h') \leq 1$ for all conjugated pairs $h,h'\in G$. Hence 
$\mathbf{0}(r) = 0 \leq \CLF(r) = \mathrm{id} \circ \CLF \circ \, \mathrm{id}(r)$, 
whereas  
$\CLF(r) \leq 1 =0+1 = a \circ \mathbf{0} \circ a(r)$, 
where $a(x) = x+1$. 
\end{example}

\begin{lemma} \label{lem:verybasic}
Two conjugator length functions of a compactly generated locally compact group $G$, induced by two (distinct) compact generating sets, are quasi-equivalent. 

In particular, the rate of growth of a conjugator length function for $G$ does not depend on the choice of compact generating set. 
\end{lemma}
\begin{proof}
This is immediate: if $S$, $S'$ are compact generating sets for $G$, then $(G,d_S)$ is bilipschitz equivalent to $(G,d_{S'})$ via the identity map; see \cite[Proposition~4.B.4(3)]{Cornulier--delaHarpe}.
\end{proof}

It is generally not true that the growth rate of $\CLF$ is a quasi-isometry invariant. Note that this is not conflicting with \Cref{lem:verybasic} as it relies on a much stronger coarse-geometric relationship, namely bilipschitz equivalence. 

In order to understand conjugator length growth, it should be clear that one needs a good understanding of the set of all elements $w\in G$ conjugating a given $h$ to a given $h'$. We call this set the \emph{{transporting} set} from $h$ to $h'$ and denote it by $\mathrm{C}_G(h,h')$. 
This allows us to write 
\[
\cl(h,h') = \min\{ \ell_S(k) \, : \, k\in\mathrm{C}_G(h,h') \}. 
\]
Those sets are inspected in \Cref{sec:coconjugation-sets}, but we take a moment to draw attention to two instructive examples.

\begin{example} \label{ex:Lukas}
The following two groups act geometrically (in the sense of \cite[Definition~4.C.1]{Cornulier--delaHarpe}) on Euclidean spaces and are thus CAT$(0)$ groups by the \v{S}varc--Milnor lemma; see \cite[Theorem~4.C.5]{Cornulier--delaHarpe}. However, they show distinct conjugator growth behavior.
\begin{enumerate}
\item Consider the infinite dihedral group $D_{\infty} \cong C_2 \ast C_2 \cong \langle s,t \mid s^2 = t^2 = 1 \rangle$. 
Elements of $D_\infty$ are either of infinite order --- in which case they are conjugate to the product $st$  and have even word length --- or are involutions --- in which case they are conjugate to $s$ or to $t$  and have odd word length. It is straightforward to show that the only conjugated pairs of $\langle st \rangle$ are $h = (st)^m$ and $h'=(st)^{-m}$, in which case $\mathrm{C}_{D_\infty}(h,h') = \{s,t\}$. Thus $\cl(h,h') = 1$ in these cases. On the other hand, one checks (e.g., using normal forms) that involutions of the form $h = (st)^m s (ts)^m$ and $h'= s$ are conjugate with $(st)^m \in \mathrm{C}_{D_\infty}(h,h')$ being a shortest transporter. (With similar situation swapping roles of $s$ and $t$.) Hence $\CLF$ grows linearly for $D_\infty$.

\smallskip

\item The following example is due to Lukas Vandeputte. Consider the connected Lie group $H = \R^2 \rtimes \SO(2)$. Note that it is compactly generated by (the underlying subset) $S = ([0,1] \times [0,1]) \times \SO(2)$. Now choose a small enough $\veps > 0$ and a positive angle $\theta \in (0, \veps) \subset \R$. Consider the unit vector $e_1 = (1,0) \in \R^2$ and the rotation matrix 
\[ R_\theta = \left( \begin{matrix} \cos(\theta) & -\sin(\theta) \\ \sin(\theta) & \cos(\theta) \end{matrix} \right). \]
Note that $\Id - R_\theta$ 
is invertible 
as $0 < \theta < \veps$. Direct computations show that the elements $h_\theta = (0, R_\theta)$ and $h_\theta'= (e_1, R_\theta)$ are conjugated in $H$, and
\[ \mathrm{C}_H(h_\theta,h_\theta') = \{ ((\Id-R_\theta)^{-1}(e_1), u) \in H \mid u \in \SO(2) \}. \]
Decreasing the constant $\veps > 0$ chosen at the beginning and hence decreasing the angle $\theta$, we see that the vector $(\Id-R_\theta)^{-1}(e_1) = \frac{1}{2-2\cos(\theta)}\cdot (1-\cos(\theta),\sin(\theta))$ is arbitrarily far away from the unit square $[0,1]\times [0,1] \subseteq \R^2$ since the multiplying factor $\frac{1}{2-2\cos(\theta)}$ increases. Hence the word length $\ell_S(((\Id-R_\theta)^{-1}(e_1), u))$ is arbitrarily large. On the other hand, $\ell_S(h_\theta) = 1$  and $\ell_S(h_\theta') = 2$. Thus, for every $n \geq 3$, by taking $\theta \to 0$ we always have conjugated pairs $h_\theta$, $h_\theta'$ yielding arbitrarily large $\cl(h_\theta,h_\theta')$. Thus, in $H$, 
$\CLF(m) = \infty$ 
for all $m \geq 3$.
\end{enumerate}
\end{example}

\subsection{Our setup --- split locally compact Euclidean isometry groups}
\label{sec:mainsetup}

Recall that the full isometry group of $\E^n$ splits as a semidirect product $\Isom(\E^n) = T \rtimes \On$, where $T \cong \R^n$ is the translation subgroup of~$\aG$ and $\On$ is the group of orthogonal transformations of $\E^n$ with the standard inner product. We view $\Isom(\E^n)$ as a locally compact group using its usual topology as a Lie group. In particular, $\On$ is compact in $\Isom(\E^n)$ and the group isomorphism $T \cong \R^n$ is topological, so that $T$ carries a Euclidean topology. 

We say a subgroup $H$ of $\Isom(\E^n)$ \emph{splits} if  $\aH =  T_\aH \rtimes \sH$ with $T_\aH = T \cap \aH$ and $\sH = H \cap \On$. In case $H$ is a split topological subgroup, we require the split to be topological as well and in particular with a continuous action. 
Translations of $\E^n$ by $\lambda \in \R^n$ are written as $t^\lambda$. 
Moreover, the action of the orthogonal part $\sH = H \cap \On$ on the subgroup of translations $T_\aH = T \cap \aH$ is compatible with the action of $\sH$ on $\R^n$. That is, $h_0 t^{\lambda} h_0^{-1} = t^{h_0(\lambda)}$ for all $h_0 \in \sH$ and all $\lambda \in L_\aH$; cf. \cite{MST5}. 
The \emph{translation lattice} of $H$ is defined as $L_\aH = \{ \lambda \in \R^n \mid t^\lambda \in T_\aH \}$ and is by definition an $H_0$-invariant submodule of the ambient space $\R^n$, and we put $V=\Span_\R(L_H)\subset \R^n$. (Despite its name, $L_\aH$ is not necessarily a lattice; see \Cref{lem:structureofH}(4) below.)
As in \cite{MST4} we will call $t^\lambda$ the \emph{translation part} and $h_0$ the \emph{spherical part} of~$h$. Accordingly, we call $T_\aH$ and $\sH$ the \emph{translation} and \emph{spherical} subgroups of $H$, respectively. 
We shall need the following observations.

\begin{lemma} 
\label{lem:structureofH} 
Let $H = T_\aH \rtimes \sH$ be a split, locally compact subgroup of $\Isom(\E^n)$. Then:
\begin{enumerate}
\item The spherical part $\sH$ is compact.
\item $H$ is compactly generated.
\item The inclusion of the translation part $T_H$ into $T$ is a quasi-isometric embedding.
\item The translation lattice $L_H$ is isomorphic to $\R^a \times \Z^b$ for some $a,b \geq 0$ with $a+b \leq n$. In particular, $L_H$ has bounded covolume in its real span $V=\Span_\R(L_H) \cong \R^{a+b}$.
\item Both the real span $V = \Span_\R(L_H) \subseteq \R^n$ and its orthogonal complement $V^\perp \subseteq \R^n$ are $H_0$-invariant subspaces of $\R^n$. 
\end{enumerate}
\end{lemma}

\begin{proof}
Since $H$ is locally compact and $\Isom(\E^n)$ is a Hausdorff space, we see that $H$ is closed; cf. \cite[Chapter~III, {\S}~3, no.~3, Corollary~2]{BourbakiTopology1-4}. Hence, the spherical subgroup $\sH \leq \On$ is also closed, therefore compact as $\On$ is so. 

Because $H$ is closed, its translation part $T_\aH \nsgp H$ is a closed (locally compact) subgroup of the abelian group $T \cong \R^n$. This implies that $T_\aH$ is compactly generated; e.g., by \cite[Proposition~5.A.7]{Cornulier--delaHarpe}. Thus the extension $H = T_\aH \rtimes \sH$ is compactly generated as well; see \cite[Proposition~2.C.8]{Cornulier--delaHarpe}. 

Closed subgroups of $\R^n$ are of the form $\R^a \times \Z^b$ for some integers $a,b \geq 0$, where the $a+b \leq n$ generators are $\R$-linearly independent vectors; see, for example, \cite[Theorem~9.22]{ADGBbook}. (In particular, the real factor $\R^a$ is connected, and the integer part $\Z^b$ is discrete in $\R^{a+b} \subseteq \R^n$.) The closed subgroup $T_\aH \leq T \cong \R^n$ is thus topologically isomorphic to such an $\R^a \times \Z^b \subseteq \R^{a+b} \subseteq \R^n$, with $\R^a \times \Z^b \qi \R^{a+b} \subseteq \R^n$. Part~(4) also follows since $V=\Span_\R(L_H) \cong \R^{a+b}$ and $L_H \cong T_H = \R^a \times \Z^b \qi \R^{a+b} \subseteq \R^n$.

The $\Z$-module $L_H = \{\lambda \in \R^n \mid t^\lambda \in T_\aH\}$ is $H_0$-invariant by definition. Since elements of $H_0$ are linear maps, $V$ is $H_0$-invariant as well. Given arbitrary vectors $v \in V$ and $w \in V^\perp$ and any $h_0 \in H_0$, one has 
$\langle h_0(w), v \rangle = \langle h_0^{-1} h_0(w), h_0^{-1}(v) \rangle = \langle w, h_0^{-1}(v) \rangle = 0$ 
because $H_0 \leq \On$, $h_0^{-1}(v) \in V$, and $w \in V^\perp$. Thus $h_0(w) \in V^\perp$ and part~(5) follows.
\end{proof}

We now introduce the notion of {essential part} of $\sH$, which will dictate the growth rate in \Cref{thm:lineargrowth} in the nontrivial cases.  
 
\begin{definition}
\label{def:essentialH0}
Let $\aH = T_\aH \rtimes \sH$ be a (not necessarily topological) 
split subgroup of $\Isom(\E^n)$ and write $V = \Span_\R(L_H)$. We call  
\(
H_0^{\mathrm{in}} := \{ h_0 \in H_0 \mid h_0|_V = \mathrm{id}_V \}
\) 
the \emph{inessential part} of $\sH$ and will refer to 
\( 
\til\sH = \{ h_0|_V \mid h_0 \in H_0 \}
\)
as the \emph{essential part} of $\sH$. 
\end{definition} 

\begin{lemma} 
\label{lem:essentialHomomorphism}
	With notation as in \Cref{def:essentialH0} above, $H_0^{\mathrm{in}} $ is normal in $\sH$ and there exists a surjective homomorphism $f : H_0 \to \til\sH$ with $\Ker(f) = H_0^{\mathrm{in}}$ and such that 
	\[
	f(h_0) t^\lambda f(h_0)^{-1} = h_0 t^\lambda h_0^{-1} = t^{h_0(\lambda)}
	\] 
	for all $h_0 \in \sH$ and $\lambda \in L_H$. 
	In particular,  $f$ extends to a surjective homomorphism 
	\[
	f : \aH = T_\aH \rtimes \sH \to \til\aH \define \til T_\aH \rtimes \til\sH,
	\] 
	with $\Ker(f) = H_0^{\mathrm{in}}$ and with image $f(H) = \til T_\aH \rtimes \til\sH$ a split subgroup of (the Lie subgroup) $\Isom(V)\leq \Isom(\E^n)$, where 
	$\til T_H$ is the restriction of $T_H$ to $V$ and $f(t^\lambda) = t^\lambda|_V$ for every $t^\lambda \in T_\aH$. Moreover, the translation lattices of $\til\aH$ and $\aH$ agree: $L_{\til\aH} = L_H$.
\end{lemma}

\begin{proof}
This is a direct consequence of part~(5) of \Cref{lem:structureofH}. Indeed, denote by $f : \sH \to \til\sH$ the obvious projection of $\sH$ onto $\til\sH$ by restrictions of all elements to $V$. 
Then $\Ker(f) = H_0^{\mathrm{in}}$ by design. 
%
It is clear that $L_\aH$ is also $\til\sH$-invariant and thus $f(h_0)(\lambda) = h_0(\lambda)$ for all $h_0 \in \sH$ and all $\lambda \in L_\aH$. Thus, taking $\til T_\aH \define \{ t^\lambda|_V \mid t^\lambda \in T_\aH \}$ we may form the semi-direct product $\til H \define \til T_\aH \rtimes \til\sH \leq \Isom(V)$ in the usual way. Note that we may (and do) identify $\Isom(V)$ as a closed subgroup of $\Isom(\E^n)$, 
as each map $g \in \Isom(V)$ extends to all of $\R^n$ by setting $g|_{V^\perp} = \mathrm{id}$. 
It is clear that the two translation lattices are the same since translations from $\til T_{\aH}$ use the exact same vectors as translations from $T_{\aH}$. 
\end{proof}

\subsection{Split Euclidean isometry groups and their transporting  sets}\label{sec:coconjugation-sets}

Let $\aH$ be as in \Cref{sec:mainsetup}. For all $h\in \aH$ write $\hconj$ for its conjugacy class and for all pairs $h,h' \in \aH$ write 
\[
\coconjH{h}{h'} = \{ k \in \aH \mid khk^{-1} = h' \}
\]
for the \emph{{transporting} set}  (from $h$ to $h'$). 
Note that  $\coconjH{h}{h} = C_H(h)$ is the centralizer of $h$ in $\aH$.
We summarize the relevant results from \cite{MST4}. 

Fix $h=t^\lambda h_0\in\aH$ and also $h'=t^{\lambda'} h'_0 \in \hconj$. The {transporting} set $\coconjH{h}{h'}$ is equal to $k \Cent_{\aH}(h)$ for any $k \in \aH$ such that $khk^{-1} = h'$. However, this description is not helpful for our purpose as one would still have to guess a conjugating element and also determine the centralizer. The theorem below provides a description of {transporting} sets (and hence also centralizers) which does not require any guessing. 
Following \cite[Definition~2.1]{MST4} we consider $\Mod_\aH(h_0') = (\Id - h_0') L_\aH$, which is itself contained in $L_\aH$; see \cite[Lemma~3.3]{MST4}. 
We then define the 
\emph{translation-compatible part} of $\coconj_{\sH}(h_0,h_0')$ to be the set 
	\begin{equation}\label{K0_conj}
		\coconj_{\sH}^{\lambda,\lambda'}(h_0,h_0') = \{ u \in \scoconjH{h_0}{h_0'} \mid \lambda' - u\lambda \in \Mod_\aH(h_0') \}. 
	\end{equation}
The {transporting} set is then a disjoint union of subsets parametrized by the translation-compatible part of the {transporting} set of their spherical parts.

\begin{thm}[{Transportation \cite[Theorem~1.13]{MST4}}]
	\label{thm:coconj}
	Let $\aH = T_\aH \rtimes \sH$ be a split group of Euclidean isometries. Let $h=t^{\lambda}h_0$ and $h' = t^{\lambda'}h_0'$ be conjugate elements of $\aH$, where $\lambda, \lambda' \in L_\aH$ and $h_0, h_0' \in \sH$. Then 
	\begin{equation}\label{eq:coconjFix}
		\coconjH{h}{h'} = \bigsqcup_{u \in \coconj_{\sH}^{\lambda,\lambda'}(h_0,h_0')} t^{\eta_{u}+ (\Fix(h_0') \cap L_\aH)} u 
	\end{equation}
	where for each $u$, the element $\eta_{u} \in L_\aH$ is a particular solution to the equation 
	\begin{equation}\label{eq:eta}
		\lambda'- u\lambda=(\Id - h_0')\eta.
	\end{equation}
\end{thm}
\noindent Geometrically, \Cref{eq:coconjFix} means that the transporting set $\coconjH{h}{h'}$ lies along translates of the fix-set $\Fix(h_0')=\Ker(\Id - h_0')$, which is orthogonal to $\Mov(h_0')\define (\Id-h_0')\R^n$. 

We refer the reader to \cite{MST5} for a detailed study of the surprisingly subtle structure of $\Mod_\aH(h_0')$ in the case of Coxeter groups. 

\subsection{The Moore--Penrose pseudo-inverse}
\label{sec:pseudoinverse}

In order to prove \Cref{thm:lineargrowth} we will have to find minimum-norm solutions to systems of equations, so let us recall some key concepts.

Let $Ax=b$ be a system of linear equations over $\R$. That is, for some $m,n \in\N$ the letter $A$ denotes a real $m\times n$ matrix, $b\in V=\R^m$, and $x=(x_1,\ldots, x_n)^{\mathtt{T}}$ a list of variables. 
In case $A$ is invertible the unique solution is given by $A^{-1}b$. 
To find a minimum-norm solution to the system $Ax=b$ in all other cases (with $A$ not invertible), we need to compute the distance between the space $\LAb$ and the origin. This is achieved using a notion of a pseudoinverse. There are many (equivalent) ways to define it. We chose to follow Definition~1.1.3 in \cite{CampbellMeyer}. 

\begin{definition}[{Moore--Penrose pseudoinverse}]
\label{def:MoorePenrose}
	Let $A$ be an $m\times n$ real or complex matrix. The \emph{pseudoinverse} $A^+$ of $A$ is the unique $n\times m$ matrix such that all of the following hold:
	\begin{enumerate}
		\item $AA^+A=A$, that is, $A^+$ maps all column vectors of $A$ onto themselves; 
		\item $A^+AA^+=A^+$, that is, $A^+$ has a weak inverse property; 
		\item $AA^+=(AA^+)^\star$, that is, $AA^+$ is Hermitian; and 
		\item $A^+A=(A^+A)^\star$, that is, $A^+A$ is also Hermitian.   
	\end{enumerate} 
\end{definition}

The pseudoinverse exists for any matrix and agrees with the usual inverse in case $n=m$ and $A$ is of full rank. It is easy to check that the image of $A^+$ is equal to the image of the conjugate transpose $A^\star$; see \cite[p.~9]{CampbellMeyer}. The next statement shows how pseudoinverses yield optimal solutions to linear systems. For a proof, see \cite[Theorem 2.1.1]{CampbellMeyer}.

\begin{prop}
\label{prop:pseudoinverse}
Let $A\in\R^{m\times n}$, $b\in\R^m$. Then $x_0\define A^+ b \in \R^n$ is the minimal least squares solution to the system of equations $Ax=b$. That is, $\| Ax_0-b \| \leq \| Ax- b\| $ for all $x\in \R^n$ and $x_0$ has smallest possible norm among all vectors satisfying this inequality.
\end{prop}

There are many methods to compute pseudoinverses in general, see e.g. \cite[Cap\'itulo~16]{ElonLinA} or \cite[Theorem~1.3.1]{CampbellMeyer}. 
We shall need the following special case: 

\begin{prop}[{\cite[Theorem~1.2.1(P7)]{CampbellMeyer}}]\label{prop:pseudoInverseViaSingularValue}
	Let $M$ be a matrix with singular value decomposition $M=U\Sigma V^\star$ with $\Sigma=(\sigma_{i,j})_{i,j}$ the diagonal matrix of all singular values $\sigma_i=\sigma_{i,i}$ of $M$ and $U$, $V$ unitary. Then the pseudo-inverse $M^+$ of $M$ is given by $M^+=V\Sigma^+U^\star$, where the pseudo-inverse  $\Sigma^+=(\hat\sigma_{i,j})_{i,j}$ is the diagonal matrix obtained by putting $\hat\sigma_{i,j} =\sigma_{j,i}$ for all $i\neq j$ and $\hat\sigma_{i,i}=\frac{1}{\sigma_{i,i}}$ for all nonzero $\sigma_{i,i}$ while keeping $\hat\sigma_{i,i}=0$ otherwise. 
\end{prop}

\section{Estimates for conjugator length and trivial growth}

Let $\aH = T_\aH \rtimes \sH \leq \Isom(\E^n)= T \rtimes \On$ be as before. By \Cref{lem:structureofH} we may fix once and for all a compact generating subset $S \subseteq H$, with which we consider the word length $\ell_S$ on $H$.  
We need to estimate the shortest $S$-length of elements in $\coconjH{h}{h'}$  for pairs of conjugated elements $h,h'$ whose lengths add to at most $m \in \N$. 
We start by relating the $S$-length to the norm (i.e. Euclidean length) of its translation part. 

\begin{lemma} \label{lem:firstapprox}
Let $\aH = T_H \rtimes \sH$ be compactly generated by $S$. Then 
there exist strictly positive constants $A,B,C,D$ depending on $\aH$ and $S$, such that 
\[A\cdot \ell_S(t^\lambda w) - B \leq \| \lambda \| \leq C\cdot \ell_S(t^\lambda w) + D\]
for all $t^\lambda \in T_H$ and all $w \in \sH$, where $\| \lambda \|$ denotes the Euclidean norm of $\lambda \in \R^n$.
\end{lemma}
\begin{proof}
By \Cref{lem:structureofH}, $\sH$ is compact, hence the inclusion map $T_\aH \into H = T_\aH \rtimes \sH$ is a quasi-isometry; see, for instance, \cite[Proposition~4.C.11]{Cornulier--delaHarpe}. We can thus find a compact generating set $S_0$ of $T_\aH$ and obtain constants 
$\quer{A},\quer{D} > 0$ and $\quer{B} , \quer{C} \geq 0$ such that
\begin{equation} \label{eq:firstQI}
\quer{A}\cdot \ell_S(t^\lambda w) - \quer{B} \leq \ell_{S_0}(t^\lambda) \leq \quer{C}\cdot \ell_S(t^\lambda w) + \quer{D} \quad \quad \forall t^\lambda \in T_H, w\in \sH.
\end{equation}
Choosing the origin $\mbf{0} \in \R^n$ as base point, the map 
given by \(t^\lambda \mapsto \mbf{0} + \lambda = \lambda\)
is a quasi-isometric embedding \(T_\aH \to \R^m \subseteq \R^{n}\) by the \v{S}varc--Milnor lemma; see \cite[Theorem~4.C.5]{Cornulier--delaHarpe}.  
By \Cref{lem:structureofH}(3), we have $T \qi \R^m$ for some $m \leq n$. Hence there are further constants 
$\til{A},\til{C} > 0$ and $\til{B}, \til{D} \geq 0$ such that 
\begin{equation} \label{eq:secondQI}
\til{A}\cdot \ell_{S_0}(t^\lambda) - \til{B} \leq \| \lambda \| \leq \til{C}\cdot \ell_{S_0}(t^\lambda) + \til{D} \quad \quad \forall t^\lambda \in T_H.
\end{equation}
Combining \Cref{eq:firstQI,eq:secondQI} we construct constants as claimed.
\end{proof}

\Cref{lem:firstapprox} motivates the following definition with several immediate corollaries. 

\begin{definition}
\label{def:ctd}
The \emph{conjugator translation norm} $\ctn(h,h')$ of a pair of conjugate elements $h,h'\in\aH$ is given by 
\[
\ctn(h,h')\define \min \{ \| \lambda \| \, : \, k=t^\lambda w\in \aH \text{ is such that } khk^{-1}=h'\}. 
\]  	
\end{definition} 

Note that the groups $H$ we consider here are all locally compact and compactly generated groups, hence the minimum is always attained and may be zero. 

\begin{corollary} \label{cor:secondapprox}
In the notation of \Cref{lem:firstapprox} we have 
\[ A \cdot \cl(h,h') - B \leq \ctn(h,h') \leq C \cdot \cl(h,h') + D.\]
\end{corollary}

\begin{corollary} \label{cor:thirdapprox}
Suppose $h = t^\lambda h_0$, $h' = t^{\lambda'} h_0' \in H$ are conjugate. Then there exists an increasing affine map $m_{\aH,S} : \Z_{\geq 0} \to \Z_{\geq 0}$ such that 
\[\ell_S(h) + \ell_S(h') \leq k \iff \|\lambda\| + \|\lambda'\| \leq m_{\aH,S}(k).\]
\end{corollary}

In analogy to the conjugator length we define a function maximizing the value of $\ctn$. 

\begin{definition}
	\label{def:CTD}
	
	The \emph{conjugator translation norm function} $\Tnorm:\N\to\Z_{\geq 0}$ of $H$ is defined by 
	\[
		\Tnorm(m) \define \sup\{ \ctn(h,h')  :  h=t^\lambda w\sim h'=t^{\lambda'}w' \, \text{ and } \, \| \lambda\| +\| \lambda'\|\leq m\}.
	\]
	Similarly to $\CLF$, we may view $\Tnorm$ as a nondecreasing map on $\R_{\geq 0}$ by setting $\Tnorm(r) = \Tnorm(\lceil r \rceil)$ for all $r \in \R_{\geq 0}$, and we study its growth rate as usual.
\end{definition}

The next lemma is immediate from \Cref{cor:secondapprox,cor:thirdapprox}.

\begin{lemma}
	\label{prop:reduce-to-translations}
	For the groups $H$ as before, the function $\Tnorm$ has the same growth rate as $\CLF$.  
\end{lemma}

The following technical lemma will be used in the proof of \cref{prop:Tnorm-lower-bound,prop:Tnorm-growth-nonisolated}.

\begin{lemma}\label{lem:conjugateElements}
Suppose $w\in\sH$ is such that $\{0\} \neq \Mod_\aH(w) = (\Id-w)L_\aH$. Then for all $\lambda' \in \Mod_\aH(w)  \subseteq L_\aH$ and all $z \in \Z$ the elements $w$ and $t^{z\lambda'} w$ are conjugate.  
\end{lemma}
\begin{proof}
Observe that if $\lambda'=0$ or $z=0$, then $t^{z\lambda'} w=w$ and there is nothing to prove. 
As $L_\aH$ is a $\Z$-module and $\Id-w$ a linear map, we have $\Span_{\Z}(\lambda') \subseteq \Mod(w)$ so that $z\lambda' \in \Mod(w) \setminus \{0\}$ for all nontrivial $\lambda'\in\Mod(w)$ and $z \in \Z \setminus \{0\}$. 

Since $\lambda' \in \Mod(w) \setminus \{0\}$ there exists, by \Cref{thm:coconj}, some $\eta \in L_\aH \setminus \{0\}$ for which $z\lambda' = (\Id-w)z\eta$. Thus, choosing any $u \in C_{\sH}(w)$, the element $k_z = t^{z\eta} u \in H$ satisfies
\begin{align*}
	k_z w k_z^{-1} = t^{z\eta} u w u^{-1} t^{-z\eta} = t^{z\eta} w t^{-z\eta} w^{-1} w = t^{(\Id-w)z\eta} w = t^{z\lambda'} w.
\end{align*}
Hence the claim. 
\end{proof}

We now determine the growth type of $\Tnorm$, starting with a lower bound in the non-commuting situation. \Cref{cor:thecaseofzerogrowth} below is a direct consequence of \Cref{prop:Tnorm-lower-bound}.  

\begin{prop} \label{prop:Tnorm-lower-bound}
Let $\aH$ be a locally compact split subgroup of $\Isom(\E^n)$. If $L_\aH \subseteq \Fix(\sH)$, then $\Tnorm$ has trivial growth for $\aH$. Otherwise, $\Tnorm$ grows at least linearly.
\end{prop}

\begin{proof}
In case $L_\aH \subseteq \Fix(\sH) = \{v \in \R^n \mid h_0(v) = v \, \, \forall h_0 \in \sH\}$, all elements of $\sH$ commute in $H$ with all elements from $T_\aH$. That is, $H$ actually splits as a direct product $H = T_\aH \times \sH$ of the abelian group $T_\aH$ with the compact group $\sH$. \Cref{def:conjugatorlengthfunction} and \Cref{ex:compactCLFzero} imply that $\Tnorm \qeq \CLF \qeq \mathbf{0}$ in this case.

Now suppose otherwise that there exists $w \in \sH\setminus\{\Id\}$ for which $L_\aH \not\subseteq \Fix(w)$. 
We aim to check that $\Tnorm$ grows at least linearly, which amounts to finding an increasing affine function that is quasi-dominated by $\Tnorm$. 
By assumption we can find some $\lambda' \in \Mod(w) \setminus \{0\}$. 
Note that $\| \lambda'\| > 0$ and that $\| z \lambda' \| = |z| \cdot \|\lambda'\|$ for all $z \in \Z$. 
As seen in \Cref{lem:conjugateElements}, for every $z \in \Z$ the elements $h_z = w$ and $h_z' = t^{z\lambda'} w$ can be conjugated to one another. 
By definition of $\Tnorm$ and using the conjugated pairs $h_z, h_z'$, 
\begin{equation} \label{eq:linear}
\Tnorm(m) \geq \ctn(h_z,h_z') = |z| \ctn(w, \, t^{\lambda'}w) \quad \text{ whenever } \quad |z| \cdot \|\lambda'\| \leq m.
\end{equation}
Fix a vector $\eta \in L_\aH \setminus \{0\}$ realizing $\ctn(w, \, t^{\lambda'}w)$. 
Since $\eta \neq 0$, there is some constant $\veps > 0$ for which $\| \eta\| \geq \veps$. 
Note that $\veps/\|\lambda'\| > 0$. Choosing $z = \left\lfloor \frac{1}{\| \lambda' \|} \cdot m\right\rfloor \in \N$, we have $z \leq \frac{1}{\| \lambda' \|} \cdot m \leq z+1$ by definition, hence $z \geq \frac{1}{\| \lambda' \|} m -1$ while also $|z| \cdot \|\lambda'\| \leq m$. In particular, we may use \Cref{eq:linear} and obtain 
\[\Tnorm(m) \geq |z| \cdot \ctn(w, \, t^{\lambda'}w) = z \cdot \|\eta\| \geq \left( \frac{1}{\| \lambda' \|} \cdot m -1 \right) \cdot \veps = \frac{\veps}{\| \lambda' \|}m -\veps.
\]
Thus $\Tnorm$ is bounded below by the increasing affine map $a(r) = \frac{\veps}{\| \lambda' \|}r -\veps$ 
and we are done.
\end{proof}

\begin{corollary} \label{cor:thecaseofzerogrowth}
A locally compact split subgroup $\aH \leq \Isom(\E^n)$ has trivial growth for its conjugator length function if, and only if, the spherical and translation parts of $\aH$ commute. In particular, the growth is trivial if  $\aH$ is abelian or compact.
\end{corollary}

\section{Groups with nontrivial growth}

Now we turn to the cases of nonzero growth and start with assuming $\sH$ is finite. 

\begin{prop}
	\label{prop:Tnorm-growth-discrete}
	Let $\aH=T_H\rtimes \sH$ be a locally compact split subgroup of $\Isom(\E^n)$ whose translation and spherical part do not commute. If 
	$\sH$ is finite, then the function $\Tnorm$ for $\aH$ grows linearly.
\end{prop}

\begin{proof}
Recall that $H$ is compactly generated by $S$. The hypothesis $\CLF \not\qeq \mathbf{0}$ implies that $\sH$ is nontrivial and the translation lattice $L_{\aH}$ contains nonzero vectors. It suffices to prove that there is an increasing affine map $M_{\aH,S} \qeq m\mapsto m$ for which $\ctn(h,h') \leq M_{\aH,S}(m)$ whenever $h = t^\lambda h_0$ and $h' = t^{\lambda'} h_0'$ are conjugate and satisfy $\|\lambda\| + \|\lambda'\| \leq m$. 
This implies that $\Tnorm(m) \qdom M_{\aH,S}(m)$, whence $\Tnorm(m)$ grows linearly on $m$ thanks to \Cref{prop:Tnorm-lower-bound}. 
So let $h = t^\lambda h_0 \sim h' = t^{\lambda'} h_0'$ be given as above, and let $g = t^{\eta_u} u$ be a transporter that realizes $\ctn(h,h')$. We need to estimate $\|\eta_u\|$.

By \Cref{thm:coconj}, we know that $\eta_u$ satisfies \Cref{eq:eta}, that is, $\lambda' - u \lambda = (\Id-h_0') \eta_u$. 
Recall that the vector $\eta_u$ is, by assumption, shortest possible in $L_\aH$. 
Now, taking $A = \Id-h_0'$ and $b = \lambda'-u\lambda$ we see that 
\(\eta_u \in \mathcal{L}_{A,b} = \{x \in \Span_{\R}(L_H) \mid Ax = b\}. \)
Apply \Cref{prop:pseudoinverse} to see that the shortest possible element in $\mathcal{L}_{A,b}$ is 
\(
\eta_0 \define A^+ b,
\) 
where $A^+$ is the Moore--Penrose pseudoinverse of $A = \Id-h_0'$. 
Note that, since $\Range{A^+} = \Range{A^\ast}$ and $b = \lambda'-u\lambda \in L_\aH$, we get
\[
\eta_0 = A^+b \in A^\ast (L_\aH) = (\Id-h_0')^\ast (L_\aH) = (\Id-(h_0')^\ast) L_\aH = (\Id-(h_0')^{-1}) L_\aH \subseteq L_\aH
\]
as $h_0' \in \sH \subseteq \On$. Thus $\|\eta_u\| = \|\eta_0\|$ because both $\eta_u$ and $\eta_0$ are shortest possible vectors in $L_\aH$ satisfying \Cref{eq:eta}.
On the other hand, 
\Cref{lem:structureofH}(1) also implies that the set 
$\Id - \sH = \{\Id - w \mid w \in \sH\}$ 
is compact. Since $\sH$ is finite, taking operator norms $|-|$ in our ambient Euclidean space $\E^n \cong \R^n$ it is clear that the set $\{ |B^+| \, : \, B \in \Id-\sH \}$ is bounded since there are only finitely many elements in $\sH$. 
This gives us a constant $k_1 > 0$ such that $|B^+| < k_1$ for all $B \in \Id-\sH$, in particular $|A^+| < k_1$ as well. 
Note also that
\[
\|b\| = \|\lambda'-u\lambda\| \leq \|\lambda'\| + |u|\cdot \|\lambda\|.
\]
Thus, taking $k_2 = k_{2,\aH,u} = \max\{1, |u|\}$, we get 
\[
\|b\| \leq k_2 \cdot (\|\lambda'\| + \|\lambda\|).
\]
A priori, $k_2$ depends on the operator norm $|u|$. But again using that $\sH$ is compact by part~(1) of \Cref{lem:structureofH}, there is a uniform upper bound on the operator norm of elements in $\sH$. In particular, taking the constant $k_3 = k_{3,\sH} \define \max\{ |w| \, : \, w \in H_0\} > 0$, we obtain
\begin{equation} \label{eq:lastconstant}
	\|b\| \leq k_3 \cdot (\|\lambda'\| + \|\lambda\|).
\end{equation}

It follows from \Cref{eq:lastconstant} that
\[
 \|\eta_u\| = \|\eta_0\| = \|A^+ b\| \leq |A^+| \cdot \|b\|\leq k_1 k_3 (\|\lambda'\| + \|\lambda\|) \leq M_\aH(m), \]
where $M_\aH(r) \define k_1 k_3 r $. This completes the proof.
\end{proof}

For an illustration of the situation in \Cref{prop:Tnorm-growth-discrete} compare \Cref{ex:conjugatornorm}.

Recall from \Cref{lem:essentialHomomorphism} that there exists a surjection from $\aH$ onto the essential part $\til\aH$ of every split group $\aH$. We prove that the two groups involved have identical conjugator translation norm. %

\begin{prop}
	\label{prop:Tnorm-growth-essential-restriction}
	Let $\aH$ be a locally compact split subgroup of $\Isom(\E^n)$. 
	Then the functions $\Tnorm$ for $\aH$ and $\Tnorm$ for $\til\aH$ agree.
\end{prop}

\begin{proof}
Recall the projection $f : H = T_\aH \rtimes \sH \onto \til\aH = \til{T_\aH} \rtimes \til\sH$ from \Cref{lem:essentialHomomorphism} induced by the restrictions to $V = \Span_\R(L_\aH)$ on both the translation and spherical part. Then $f$ has $\ker(f) = H_0^{\mathrm{in}}$. 
Write $\Tnorm_{\til\aH}$ for the conjugator translation norm function of $\til\aH$ and similarly $\Tnorm_\aH$ for the one for $\aH$. By definition, $\Tnorm_{\til\aH}$ and $\Tnorm_{\aH}$ are suprema over values of the conjugator translation norm function $\ctn$. 
Thus we have to first analyze conjugator translation norms $\ctn(a,a')$ and $\ctn(h,h')$ for $a,a'\in\til\aH$, respectively $h,h'\in\aH$. 

Any conjugate pair $h \sim h'$ in $H$ gives rise to a conjugate pair $f(h) \sim f(h')$ in $\til\aH$. By definition, $f(t^\lambda) = t^{\lambda}|_V$ for all $t^\lambda \in T_\aH$, so that the translation vectors for $h$ and $f(h)$ are exactly the same and in particular have equal length. Moreover, $\ctn(f(h), f(h')) \leq \ctn(h,h')$. To see this, let $h=t^\lambda h_0$ and $h'=t^{\lambda'} h_0'$ be given and suppose $khk^{-1}= h'$ for some $k=t^\mu u\in\aH$ that realizes $\ctn(h,h')$. Then $f(k)$ conjugates $f(h)$ onto $f(h')$. As $f(k)$ also has translation vector $\mu$, we obtain 
	\begin{align*}
	\ctn(f(h), f(h')) 
	 & = \min \{ \| \xi \| \, : \, b=t^\xi\vert_V w\vert_V\in \til\aH \text{ and } bf(h)b^{-1}=f(h')\} \leq \|\mu\|=\ctn(h,h') 
	\end{align*} 	

We now establish an adequate converse to the above. Due to surjectivity of $f$, any two $a, a' \in \til\aH$ are images of elements $h, h' \in \aH$ under $f$. In case $a,a'$ are conjugated in the image $\til\aH=f(\aH)$, it may or may not happen that $h \sim h'$ in $\aH$. Nevertheless, we claim that for every conjugated pair $a \sim a'$ in $\til\aH$ there exist $h \in f^{-1}(\{a\})$ and $h' \in f^{-1}(\{a'\})$ for which $h \sim h'$ in $\aH$ and such that the translation vectors of $a$ and $h$ (resp. of $a'$ and $h'$) are equal, and moreover $\ctn(a,a') = \ctn(f(h), f(h')) \geq \ctn(h,h')$. 

In effect, let $a,a' \in \til\aH$ be conjugated, write $a=t^\lambda|_V u |_V$ and $a'=t^{\lambda'}|_V u'|_V$ with $\lambda, \lambda' \in L_H$, and $u,u' \in \sH$. Now let $b \in \til\aH$ be an element realizing $\ctn(a,a')$. That is, $b a b^{-1} = a'$ and there are some $\mu \in L_H$, $w \in \sH$ such that $b = t^{\mu}|_V w|_V$ and $\|\mu\| = \ctn(a,a')$. Expanding on the equality $bab^{-1}=a'$, we obtain 
	\begin{align*}
	a' & = t^{\lambda'}|_V \, u'|_V = bab^{-1} = t^{\mu}|_V \, w|_V \, t^\lambda|_V \, u|_V \, w^{-1}|_V \, t^{-\mu}|_V \\ 
	& = t^{\mu + w(\lambda) - wuw^{-1}(\mu)}|_V \, w|_V \, u|_V \, w^{-1}|_V,
	\end{align*}
	which --- by uniqueness of expression in the semi-direct product --- yields 
	\begin{equation} \label{eq:conjHtilH}
	\begin{cases} 
		{w|_V \, u|_V \, w^{-1}|_V  = (u')|_V, }\\
		\mu + w(\lambda) - wuw^{-1}(\mu) = \lambda'.
	\end{cases}
	\end{equation}
	Recall that all translation vectors $\lambda, \lambda', \mu$ lie in $L_\aH \subseteq V = \Span_\R(L_\aH)$. The first equality in Equation~\eqref{eq:conjHtilH} shows that there exists some $x \in \ker(f) = H_0^{\mathrm{in}}$ for which $w u w^{-1} = xu' \text{ in } \sH$. 
	Now define $h = t^{\lambda} u$ and $h' = t^{\lambda'} xu'$. Note that 
	\[f(h) = a \quad\text{ and }\quad f(h') = t^{\lambda'}|_V \, x|_V \, u'|_V = t^{\lambda'}|_V \, \mathrm{id}_V \, u'|_V = t^{\lambda'}|_V \, u'|_V = a'.\] 
	By design, the translation vectors for $h, h'$ are $\lambda$ and $\lambda'$, respectively. Moreover, $h$ and $h'$ are conjugated to one another via the element $k \define t^{\mu} w \in \aH$. 
	Thus 
	\[
	\ctn(f(h),f(h')) = \ctn(a,a') = \|\mu \| \geq \min\{ \|\eta\| \, : \, l=t^{\eta} y\in \aH \text{ and } lhl^{-1}=h'\} = \ctn(h,h').
	\]

	The findings of the previous paragraphs show that a pair $a,a' \in \til\aH$ is conjugated (in $\til\aH$) if and only if they are images of a pair $h,h' \in \aH$ that is conjugated (in $\aH$). Moreover, taking suprema over the inequalities obtained above for the functions $\ctn$, we obtain  
	\begin{align*}
	\Tnorm_{\aH}(m) 
	& = \sup\{ \ctn(h, h') \, : \, h=t^\lambda \, h_0 \, \sim \,  h_0'=t^{\lambda'} h_0' \, \text{ in $\aH$,  and } \, \| \lambda\| +\| \lambda'\|\leq m\} \\
	& \geq \sup\{ \ctn(f(h), f(h')) \, : \, h=t^\lambda \, h_0 \, \sim \,  h_0'=t^{\lambda'} h_0' \, \text{ in $\aH$,  and } \, \| \lambda\| +\| \lambda'\|\leq m\} \\
	& = \sup\{ \ctn(a, a') \, : \, a=t^\lambda|_V \, u|_V \, \sim \,  a'=t^{\lambda'}|_V \, u'|_V \, \text{ in } \til\aH, \, \| \lambda\| +\| \lambda'\|\leq m  \text{ and } \\
 	& \phantom{\sup\{ \ctn(a, a') : } \, \, \quad  \exists h = t^\lambda h_0 \in f^{-1}(\{a\}), \, h' = t^{\lambda'} h_0' \in f^{-1}(\{a'\}) \text{ with }  h \sim h' \text{ in } \til\aH\}\\ 
	& = \sup\{ \ctn(a, a') \, : \, a=t^\lambda|_V u|_V \, \sim \,  a'=t^{\lambda'}|_V u'|_V \, \text{ in } \til\aH, \, \| \lambda\| +\| \lambda'\|\leq m\} = \Tnorm_{\til\aH}(m)  \\
	& = \sup\{ \ctn(f(h), f(h')) \, : \, a=f(t^\lambda h_0) \, \sim \,  a'=f(t^{\lambda'}h_0') \, \text{ in } \til\aH, \, \| \lambda\| +\| \lambda'\|\leq m\} \\
	& \geq \sup\{ \ctn(h, h') \, : \, h=t^\lambda h_0 \, \sim \,  h'=t^{\lambda'}h_0' \, \text{ in } \aH, \, \| \lambda\| +\| \lambda'\|\leq m\} = \Tnorm_{\aH}(m).
	\end{align*} 
That is, $\Tnorm_{\aH}(m) = \Tnorm_{\til\aH}(m)$, as desired.
\end{proof}

\begin{corollary} \label{cor:thecaseoflineargrowth}
Let $\aH=T_H\rtimes \sH$ be a locally compact split subgroup of $\Isom(\E^n)$ whose translation and spherical part do not commute. If $\til\sH$ is discrete, then the conjugator length function of $\aH$ grows linearly. 
\end{corollary}
\begin{proof}
As before it suffices to show that $\Tnorm_H$, the conjugator translation norm for $H$, grows linearly. Since $T_\aH$ and $\sH$ do not commute, \Cref{cor:thecaseofzerogrowth} together with \Cref{prop:Tnorm-lower-bound} show that $\Tnorm$ grows at least linearly. By \Cref{prop:Tnorm-growth-essential-restriction}, we have $\Tnorm_\aH(m) = \Tnorm_{\til\aH}(m)$ for every $m$, where $\Tnorm_{\til\aH}$ is the conjugator translation norm function of $\til\aH$. By construction (cf. \Cref{lem:essentialHomomorphism}), $\til\aH = \til{T_\aH} \rtimes \til\sH$ itself is a (locally compact, split) subgroup of $\Isom(V)$, where $V=\Span_\R(L_\aH)$ is a Euclidean subspace of $\E^n$. In particular, its spherical part $\til\sH$ is compact, hence finite as we are assuming it to be discrete. Applying \Cref{prop:Tnorm-growth-discrete} to $\til\aH$, we conclude that $\Tnorm_{\til\aH}$ grows linearly, hence so does $\Tnorm_\aH$. 
\end{proof}

It thus remains to consider the case where the essential part $\widetilde{\sH}$ is nondiscrete. 

\begin{prop}
	\label{prop:Tnorm-growth-nonisolated}
	Let $\aH = T_\aH \rtimes \sH$ be a locally compact split subgroup of $\Isom(\E^n)$ whose conjugator length function has nonconstant growth rate. 
	If $\til\sH$ is nondiscrete, then the function $\Tnorm$ for $\aH$ is unbounded, i.e., $\Tnorm(m) = \infty$ for all but finitely many $m \in \N$. 
\end{prop}

\begin{proof}
	Our goal is to construct a collection of conjugated pairs $h_s \sim h_s'$ for which $\ctn(h_s, h_s')$ is bounded below by a positive value that grows with $s$, yielding unboundedness of $\Tnorm$. 
	
	Since $\aH$ has nonzero growth for its conjugator length function, we know from \Cref{cor:thecaseofzerogrowth} that $T_\aH$ and $\sH$ do not commute. In particular, there must exist at least some element $w \in \sH \setminus \{\Id\}$ acting nontrivially on the translation lattice $L_\aH$, i.e., $\Mod_\aH(w) \neq \{0\}$.
	
	We claim that something much stronger holds: for every open neighborhood $\mathcal{U} \subseteq \sH$ of the identity $\id_{\sH}=\Id$, there is some nontrivial element $w_{\mathcal{U}} \in \mathcal{U} \setminus \{\Id\}$ for which $\Mod_\aH(w_{\mathcal{U}})=(\Id-w_{\mathcal{U}}) L_\aH \neq \{0\}$. In effect, if this were not the case, then the identity $\Id$ would have an open neighborhood $\mathcal{U}$ all of whose elements act trivially --- i.e., as the identity --- on the translation lattice $L_\aH$. 
	Without loss of generality, we may assume that this $\mathcal{U}$ is diffeomorphic to an open ball of dimension $\dim(\sH)$ since $\sH$ is a (closed) Lie subgroup of $\On$. 
	Consider the subgroup $K \define \langle \mathcal{U} \rangle \leq \sH \leq \On$. Note that it is connected since $\mathcal{U}$ (hence $K$) is contained in the connected component $K^\circ$ of the identity in $K$ (whereas $K^\circ \subseteq K$ by definition), so that $K=K^\circ$. Since $\dim(K) = \dim(\sH) = \dim(\sH^\circ)$, where $\sH^\circ$ is the connected component of $\Id\in \sH$, we conclude that $K = \sH^\circ$. 
	Now observe that $K = \sH^\circ$ is contained in the inessential part $\sH^{\mathrm{in}} = \{ h_0 \in \sH \mid h_0|_V = \mathrm{id}_V \}$ because $V = \Span_\R(L_\aH)$ and all elements of $K=\langle\mathcal{U}\rangle$ are linear maps acting trivially on $L_\aH$. Thus $\sH / \sH^\circ$ surjects onto $\sH / \sH^{\mathrm{in}} \cong f(\sH) = \til\sH$. Since the connected component of the identity in a compact Lie group has finite index, it follows that $\til\sH$ is finite, hence discrete, contradicting our assumptions. 
	
	To construct our sequence of conjugated pairs, we need suitable elements from a refined open neighborhood of the identity. Consider the Lie algebra $\mathfrak{h}_0$ of $\sH$ and choose a small enough open neighborhood $\mathcal{U}$ of $\Id \in \sH$ that is diffeomorphic, via the exponential map $e : \mathbb{B} \to \mathcal{U}$, to an arbitrarily small (Euclidean) open ball $\mathbb{B} \subseteq \mathfrak{h}_0$ around the zero vector. We may then find, as seen in the previous paragraph, some nonzero matrix $X \in \mathbb{B} \subseteq \mathfrak{h}_0$ whose image $e^X \in \mathcal{U} \setminus \{\Id\} \subseteq \sH$ fulfills the condition $(\Id - e^X)L_H \neq \{0\}$. Hence there is some nonzero vector $\lambda' \in \Mod_\aH(e^X) = (\Id - e^X)L_H \subseteq L_\aH$, which gives us an element $t^{\lambda'} \in T_\aH$. We may thus consider the elements
	\[ 
	h_s \define w_s \quad \text{ and } \quad h_s' \define t^{\lambda'} w_s \quad \text{ in } \quad \aH.
	\] 

	Now look at the one-parameter group $\{e^{rX} \mid r \in \R\}$. By further shrinking $\mathcal{U}$ and $\mathbb{B}$, if necessary, the straight path $(1-s)X$ with $s \in [0,1]$ starts at $X$, touches $0$ only at the endpoint $s=1$, and lies entirely in $\mathbb{B}$. Hence its image $w_s \define e^{(1-s)X}$ lies entirely in $\mathcal{U} \subseteq \sH$ and meets the identity only at $w_1 = e^0 = \Id$. In particular, $|\Id-w_s| \neq 0$ when $s \in [0,1)$ and moreover
	\[
	\lim_{s \to 1} w_s = \Id \quad \text{ or, equivalently, } \quad |\Id - w_s | \to 0 \text{ as } s \to 1. 
	\]

	Using the same $\lambda' \in \Mod_\aH(e^X) \setminus \{0\}$ found before and taking pseudoinverses, define the vectors $\eta_s \define (\Id-w_s)^+\lambda'$ for each $s \in [0,1]$. Because $\Range{(\Id-w_s)^+} = \Range{(\Id-w_s)^\ast} = \Range{(\Id-w_s^{-1})}$ and $\lambda'\in L_\aH$, we have that $\eta_s \in L_\aH$. Thus $t^{\eta_s} \in T_\aH$ for all $s \in [0,1]$.

	By (the proof of) \Cref{lem:conjugateElements}, elements of the form $k = t^{\eta_s} u \in H$ with $u \in C_{\aH}(w_s)$ conjugate $h_s$ to $h_s'$. And recall from \Cref{thm:coconj} that any $t^{\zeta} v \in C_{\aH}(h_s, h_s')$ must satisfy \Cref{eq:eta}, so that $\lambda' = (\Id - w_s) \zeta$. But by \Cref{prop:pseudoinverse} and by design, the shortest possible vector satisfying \Cref{eq:eta} is exactly $\eta_s = (\Id-w_s)^+\lambda'$. Therefore 
	\[
	\ctn(h_s, \, h_s') = \| \eta_s \| = \|  (\Id-w_s)^+\lambda' \|.
	\]

	Finally, we observe that $\|\eta_s\|$ (and thus $\ctn(h_s, \, h_s')$) goes to $\infty$ as $s$ goes to $1$. This is clear because, as $s \to 1$, the vector $\lambda' \neq 0$ remains constant whereas the operator norm of $(\Id-w_s)^+$ increases. This can also be checked more explicitly computing norms with coordinates: For instance, for each $s \in [0,1)$ let $\Sigma_s$ denote the scaling part of the singular value decomposition of the orthogonal matrix $w_s \in \sH \subseteq \On$, where we are organizing the singular values in such a way that those valued 1 are all at the bottom right of the diagonal. That is, 
	\[
	\Sigma_s = \mathrm{diag}(\sigma_{s,1}, \ldots, \sigma_{s,r_s}, 1, \ldots, 1 )
	\text{ and } w_s = V_s \Sigma_s V_s^\star,
	\]
	where $V_s$ is a unitary matrix 
	chosen such that $\sigma_{s,k}\neq 1$ for all $k\leq r_s$. As $w_s$ converges to $\Id$, the entries $\sigma_{s,k}$ of $\Sigma_s$ must all converge to $1$. Noting that $(I-w_s)$ is a normal operator --- since $w_s$ is normal and commutes with the (normal) identity operator $\Id$ --- we compute $V_s (\Id-w_s)V_s^\star = V_s V_s^\star - V_s w_s V_s^\star = \Id- \Sigma_s$. That is, the scaling matrix of the singular value decomposition of $(\Id-w_s)$ is just $A_s\define \Id-\Sigma_s$, using the exact same unitary change-of-basis matrix $V_s$ as for $w_s$. We may then combine this with \Cref{prop:pseudoInverseViaSingularValue} to compute 
	\begin{equation*} 
		(I-w_s)^+ = V_s A_s^+ V_s^\star, \text{ where } 
		A_s^+ = \mathrm{diag}\left( \frac{1}{(1 - \sigma_{s,1})}, \ldots,  \frac{1}{(1 - \sigma_{s,r_s})}, 0, \ldots, 0\right).
	\end{equation*} 
	This shows, after applying a base-change using unitary matrices (if needed) and hence not affecting the norms of vectors, that the pseudoinverse of $(I-w_s)$ is given by the diagonal matrix $A_s^+$. By construction, since each $\sigma_{s,k}$ converges to $1$ as $s \to 1$, all non-one entries of each $A_s^+$ diverge to $\infty$ as $s \to 1$. Thus (up to change of basis) the matrix $(\Id-w_s)^+$ rescales some nonzero coordinates of $\lambda'$ by factors of the form $\frac{1}{(1 - \sigma_{s,i_s})}$ which explode to $\infty$ as $s$ goes to $1$. Therefore $\lim_{s\to 1} \|\eta_s\| = \lim_{s\to 1} \|(\Id-w_s)^+\lambda'\| = \infty$.

	It then follows, for every $m \in \N$ satisfying $ m \geq \|\lambda'\|$, that 
	\[ 
	\Tnorm(m) \geq \ctn(h_s, \, h_s') {\to} \infty \text{ as } s \to 1
	\]
	by definition of $\Tnorm$ and by the previous paragraphs. 
\end{proof}

We conclude with a remark and an example. 

\begin{remark}
	Bridson and Haefliger in \cite{BridsonHaefliger} show that the growth of conjugator length functions of $\mathrm{CAT}(0)$-groups is at most exponential. This seemingly contradicts the unbounded Case~\eqref{1notisolated} of \Cref{thm:lineargrowth}. However, \cite{BridsonHaefliger} consider finitely generated (topologically discrete) groups acting isometrically, properly discontinuously and cocompactly on a $\mathrm{CAT}(0)$-space. In this paper we consider locally compact Hausdorff groups --- which may or may not be discrete. In our setting a `geometric' action is defined slightly differently; we refer the reader to \cite[Definition~4.C.1]{Cornulier--delaHarpe} for a precise definition. Overlap with the situation in \cite{BridsonHaefliger} only exists for Cases~\eqref{trivialGrowth} and~\eqref{1isolated} in which we specify the bound and compute the precise growth type of $\CLF$. The (seemingly contradictory) unbounded Case~\eqref{1notisolated} only concerns groups not covered by the results in \cite{BridsonHaefliger}. 
\end{remark}

\begin{figure}[htb]
	\centering
	\begin{overpic}[width=0.6\linewidth, trim={19ex 12ex 10ex 0},clip]{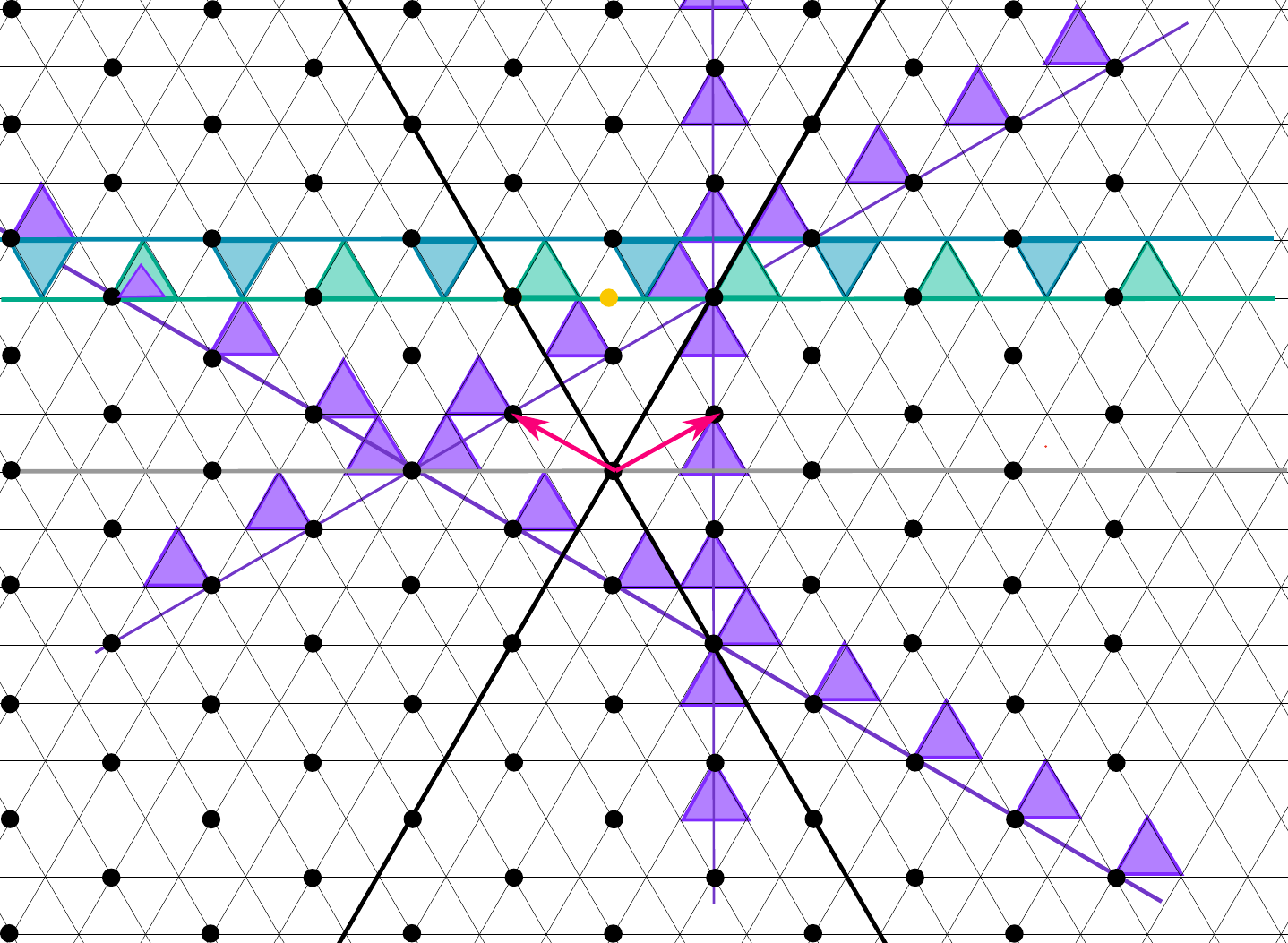}
		\put(35,53){\small{$w$}}
		\put(30,49){\small{$\eta_u$}}
		\put(35,35){\small{\textcolor{magenta}{$\alpha_2$}}}
		\put(47,35){\small{\textcolor{magenta}{$\alpha_1$}}}
		\put(42,38){\small{$\id$}}	
		\put(38,47){\small{$h$}}
		\put(52,70){\small{$h'$}}
		\put(42,53){\small{$\eta_0$}}
	\end{overpic}
	\caption{This figure illustrates \Cref{ex:conjugatornorm} and shows, among other things, the conjugacy class of the elements labeled $h$ and $h'$ (purple triangles) in a Coxeter group of type $\widetilde{\mathtt{A}}_2$.  Each turquoise element conjugates $h$ to $h'$. 
	This figure first appeared in \cite{MST5}, Example~1.1.}
	\label{fig:conjugatornorm}
\end{figure}

\begin{example}
\label{ex:conjugatornorm}
    Suppose $W$ is a Coxeter group of type $\widetilde{\mathtt{A}}_2$ generated by $S=\{s_0, s_1, s_2\}$. Then $H=W$ has spherical part isomorphic to the symmetric group on three letters (generated by $s_1$ and $s_2$), and translation part isomorphic to the group of translations by the coroot lattice $L_H\cong R^\vee$. This group acts on $\R^2$ and induces there a tiling by equilateral triangles as shown in \Cref{fig:conjugatornorm}. Each triangle corresponds to a unique group element. The triangle labeled $\id$ represents the identity. The three neighboring triangles sharing an edge with the identity triangle corresponds to the standard generators $s_i$ with $s_0$ above $\id$, the element $s_1$ to its left and $s_2$ to the right. The lattice points of the coroot lattice $L_H=R^\vee$ are marked with bold black dots in the same picture. 

    Let $h=t^\lambda s_1$ with $\lambda={\alpha_1^\vee+\alpha_2^\vee}$. The corresponding triangle in the picture is labeled with $h$. The purple elements are all conjugate to $h$. Put $h'=t^{\lambda'}w_0$ with $w_0=s_1s_2s_1$ and $\lambda'=4\alpha_1^\vee + 3\alpha_2^\vee$. The turquoise elements (both shades) are in the \emph{transporting} set $\coconjH{h}{h'}$, where we label the shortest such element by $w$. One can see that these elements are aligned on an affine translate of the fix-set of the spherical part $s_1s_2s_1$ of $h'$. The spherical part  $u$ of $w$ is $s_2$ here and  $\eta_u$ is the black vertex of the triangle labeled $w$. The yellow point corresponds to the real solution $\eta_0$ to the equation \Cref{eq:eta} and hence realizes the conjugator translator norm $\ctn(h,h')$ in the semidirect product of $\sH$ with the full group of translations $T\sim \R^n$.  
\end{example}


\renewcommand{\refname}{Bibliography}
\bibliography{bibliographyConjLength}
\bibliographystyle{alphaurl}

\end{document}